\documentclass[12 pt]{amsart}
\usepackage{hyperref}

\usepackage[disable]{todonotes}%use this to hide all todos
\usepackage{amsxtra}
\usepackage{geometry}
\usepackage{amssymb}
\usepackage{cite}
\usepackage{graphicx}
\usepackage[mathscr]{eucal}
\usepackage[normalem]{ulem}
\usepackage{tensor}
\RequirePackage{stix}

\usepackage{thmtools} 
\usepackage{thm-restate}

\usepackage{scalerel}

\usepackage{bm}

\usepackage{enumerate}
\newenvironment{enum}{  
\begin{enumerate}[\upshape(\arabic{section}.\arabic{equation}a)] }  
{  \end{enumerate}   }
\newcommand{\itemref}[2] {{\upshape(\ref{#1}\ref{#2})}}
\theoremstyle{plain}
\newtheorem{Thm}{Theorem}

\newtheorem*{MThm-A}{Theorem A}
\newtheorem*{MThm-B}{Theorem B}

\newtheorem{Cor}[Thm]{Corollary}
\newtheorem{Prop}[Thm]{Proposition}
\newtheorem{Lem}[Thm]{Lemma}

\newtheorem{Add}[Thm]{Addendum}

\theoremstyle{definition}
\newtheorem{Remark}[Thm]{Remark}

\renewcommand{\Re}{\operatorname{Re}}%real part
\renewcommand{\bar}{\overline}%bar (wide version often looks better)
\renewcommand{\tilde}{\widetilde}%tilde (wide version often looks better)

\newcommand{\C}{\mathbb C}%complex #s
\newcommand{\R}{\mathbb R}%real #s
\newcommand{\CP}{\mathbb{CP}}%complex projective space
\newcommand{\dee}{\partial}%dee
\newcommand{\vf}[1]{\dfrac{\dee}{\dee #1}}
\newcommand{\st}{\,:\,}%such that (can switch in vertical bar instead)

\newcommand{\dual}{\mathscr D}

\newcommand{\inc}{\mathscr I}

\DeclareMathOperator*{\Span}{Span}%span
\newcommand{\lam}{\lambda}%lambda
\newcommand{\mapdef}[4]{ #1 &\to #2 \\ #3 &\mapsto #4 }%specify source/target/input/output
\newcommand{\bndry}{b}%boundary
\newcommand{\w}{\wedge}%wedge
\newcommand{\inv}{^{-1}}%inverse
\newcommand{\rest}{\big\rvert}%restriction
\numberwithin{equation}{section}

\newcommand{\Rt}{\Upsilon }
\newcommand{\dpt}{\tilde{d'} }
\newcommand{\dppt}{\tilde{d''}   }
\newcommand{\Xb}{\mathcal X}
\newcommand{\Tb}{\mathcal T}
\newcommand{\Rtb}{\mathcal Y}

\newcommand{\Tt}{\tilde T}
\newcommand{\Xt}{\tilde X}
\newcommand{\Ltd}{\tilde L}
\newcommand{\Th}{\widehat T}
\newcommand{\Xh}{\widehat X}

\begin{document}
\title[Sums of CR  and projective dual CR functions]{Sums of CR  and projective dual CR functions}
\author[David E. Barrett]{David E. Barrett}
\address{Department of Mathematics\\University of Michigan
\\Ann Arbor, MI  48109-1043  USA }
\email{barrett@umich.edu}

\author[Dusty E. Grundmeier]{Dusty E. Grundmeier}
\address{Department of Mathematics \\
Harvard University \\
 Cambridge, MA 02138-2901  USA }
\email{deg@math.harvard.edu}


\thanks{The first author was supported in part by NSF grant number  DMS-1500142.}

\dedicatory{To Joseph J. Kohn}

\date{\today}

\begin{abstract}  
A smooth, strongly $\C$-convex, real hypersurface $S$ in $\CP^n$ admits a projective dual CR structure in addition to the standard CR structure. Given a smooth function $u$ on $S$, we provide characterizations for when $u$ can be decomposed as a sum of a CR function and a dual CR function. Following work of Lee on pluriharmonic boundary values, we provide a characterization using differential forms. We further provide a characterization using tangential vector fields in the style of Audibert and Bedford.
\end{abstract}

\maketitle

\section{Introduction}\label{S:intro}

A smooth real hypersurface $S$ in complex projective space $\CP^n$ is 
{\em strongly $\C$-convex} if it is locally projectively equivalent to a strongly convex
hypersurface.  (Such $S$ are automatically strongly pseudoconvex. See [Bar, \S5] for equivalent characterizations.  We do not automatically assume $S$ to be compact.)

For $p\in S$ we let $H_pS=T_pS\cap JT_pS$, the maximal complex subspace of $T_pS$.  (Here $J\colon T_p\CP^n\to T_p\CP^n$ is the complex structure tensor.)

In addition to the standard CR structure, $S$ admits a  {\em projective dual} CR structure:  if 
\begin{equation}\label{E:no-hit-0}
\text{no complex tangent hyperplane for $S$ passes through the origin}
\end{equation}
this may be defined as the unique CR structure for which the functions 
\begin{equation} \label{E:wj}
w_j(z)=\frac{\frac{\dee\rho}{\dee z_j}}{z_1\frac{\dee\rho}{\dee z_1}+\dots+z_n\frac{\dee\rho}{\dee z_n}}\qquad\quad(j=1,\dots,n)
\end{equation}
are CR, where $\rho$ is a defining function for $S$ (so $\rho$ is $\R$-valued with $\Omega=\{z\colon  \rho(z)<0\}$ and $d\rho\ne 0$ along $\bndry D$).  (Note that the values of the $w_j$ along $S$ will not depend on the choice of $\rho$.) The structure defined by this condition is projectively-invariant; along with a localization argument it follows that this construction induces a projectively-invariant CR structure on all of $S$ even when \eqref{E:no-hit-0} fails.  (See [Bar, \S6], [BG, \S3] and [BE, \S4] for more detail.)

The two CR structures share the same maximal complex subspaces ([Bar,  \S6], [APS, \S2.5]).

Given a smooth function $u$ on $S$, the goal of the current paper is to characterize whether $u$ can be decomposed as the sum of a CR function and a dual CR function. The projective decomposition problem is a natural analogue of the problem of attempting to decompose a function as a sum of a CR and conjugate-CR function, that is, of characterizing traces of pluriharmonic functions.  We prove characterizations in terms of tangential vector fields (see section \ref{S:n=2-projiharm} for precise definitions of the vector fields $X$ and $T$ and section \ref{S:n>2-projiharm} for precise definitions of $X_{jk}$ and $\Tt_{jk\ell}$).

\begin{restatable}{MThm}{thma}
\label{T:PDE-for-decomp}
For $S\subset\C^n \: (n=2)$ strongly $\C$-convex and simply-connected, the following conditions on smooth $u\colon S\to\C$ are equivalent:
\refstepcounter{equation}\label{N:PDE-for-decomp}
\begin{enum}
\item $u$ decomposes as a sum $f+g$ where $f$ is CR and $g$ is dual-CR; 
\item $XXTu=0=TTXu$. \end{enum}
\end{restatable}

This result extends the main projective decomposition theorem of \cite{BG} to non-circular hypersurfaces. 

In higher dimensions, we give a second order vector field condition. We  need to introduce the following additional condition:
\begin{equation} \label{E:star}
\text{ $z_jw_j + z_kw_k\ne 0$ for all %distinct 
$j,k$.} \tag{$\star$}
\end{equation} 
(In particular, all $z_j$ and $w_j$ are non-zero.)

\begin{restatable}{MThm}{thmb}
\label{T:Aud} For $S\subset\C^n \: (n>2)$ strongly $\C$-convex and simply-connected
and satisfying \eqref{E:star} the following conditions on smooth $u\colon S\to\C$ are equivalent.
\refstepcounter{equation}\label{N:Aud-remix}
\begin{enum}
\item  $u$ decomposes as a sum $f+g$ where $f$ is CR and $g$ is dual-CR; \label{I:decomp-hd}
\item for all distinct $j,k,\ell$ we have
\begin{equation}\label{E:Aud1}
X_{jk}\Tt_{jk\ell}u = 0. 
\end{equation}
\end{enum}
\end{restatable}

\smallskip

The condition \eqref{E:star} allows for the relatively straightforward statement of \eqref{E:Aud1}, but when it fails we will see in Proposition \ref{P:cov} that it can be repaired (at least locally) by a linear change of variable, leading to a slightly less elegant version of \eqref{E:Aud1}.

The paper is organized as follows. In \S \ref{S:d'd''}, we adapt Lee's characterization of CR pluriharmonic functions from \cite{Lee} to the projective decomposition problem.  In \S \ref{S:n=2-projiharm}, we give a vector field characterization in two dimensions, and we prove Theorem A.
%\ref{T:PDE-for-decomp}. 
In \S \ref{S:alt}, we provide an alternate construction for the vector fields in Theorem A.
%\ref{T:PDE-for-decomp}. 
In \S \ref{S:n>2-projiharm}, we set up a vector field characterization in 3 or more dimensions, and we prove Theorem B. 
%\ref{T:Aud}. 
Finally in \S \ref{S:plh}, we conclude by reviewing the pluriharmonic boundary value problem. In particular, we show the results on the sphere are remarkably similar to our projective decomposition results.  

\bigskip

\section{Operators $d'$ and $d''$}\label{S:d'd''}

For $S$ as above define an {\em $H$-form of degree $k$} on $S$ to be a smoothly-varying $\C$-valued alternating $k$-tensor on  each $H_p S$.

\begin{Prop}\label{P:d'-d''}
For smooth $u\colon S\to\C$ there are uniquely-determined degree one $H$-forms $d'u$ and $d''u$ satisfying
\begin{itemize}
\item $du\rest_H=d'u + d''u$;
\item $d'u$ is $\C$-linear with respect to the standard CR structure on $S$;
\item $d''u$ is $\C$-linear with respect to the projective dual CR structure on $S$.
\end{itemize}
The operators $d'$ and $d''$ are linear.
\end{Prop}

\begin{Lem}\label{L:J-J*}
The standard complex structure tensor $J\colon H_pS\to H_pS$ and the corresponding projective dual tensor $J^*\colon H_pS\to H_pS$ satisfy
$
\ker(J^*-J)=\{0\}.
$
\end{Lem}

\begin{proof}[Proof of Lemma \ref{L:J-J*}]
Working locally we may assume after possible application of a projective automorphism that 
\eqref{E:no-hit-0} holds so that we have a local diffeomorphism [Bar, Thm.\,16]
\begin{align*}
\mapdef{\dual_S\colon S}{\C^n}{\left(z_1,\dots,z_n\right)}{\left(w_1(z),\dots,w_n(z)\right)}
\end{align*}
with
\[J^*=\left(\dual'(p)\right)\inv\circ J\circ \dual'(p).\]

Quoting from [Bar, \S6.4] we may choose projective transformations $\chi_1,\chi_2$ so that
\begin{align*}
\chi_1(0)&=p\\
\chi_2\left(\dual_S(p)\right)&=0\\
T_0\left( \chi_1\inv(S)\right)&=\C^{n-1}\times\R\\
\intertext{and}\\
\left(\chi_2\circ\dual_S\circ\chi_1\right)'(0)\colon 
\begin{pmatrix}
z_1 \\ \vdots \\ z_{n-1} \\ u
\end{pmatrix}
&\mapsto 
\begin{pmatrix}
2i\beta_1z_1 + 2i\alpha_1\bar{z}_1\\ \vdots \\ 2i\beta_{n-1}z_{n-1} + 2i\alpha_{n-1}\bar{z}_{n-1} \\ -u
\end{pmatrix}
\end{align*}
with $0\le\beta_j<\alpha_j$ (in fact $\alpha_j^2-\beta_j^2=1/4$).

  If $\chi_1'(0)\cdot\begin{pmatrix}
z_1 \\ \vdots \\ z_{n-1} \\ 0
\end{pmatrix}\in \ker(J^*-J)$ then 
\begin{equation*}
\left(\chi_2\circ\dual_S\circ\chi_1\right)'(0) \cdot \begin{pmatrix}
iz_1 \\ \vdots \\ iz_{n-1} \\ 0
\end{pmatrix}
= i \left(\chi_2\circ\dual_S\circ\chi_1\right)'(0)\cdot\begin{pmatrix}
z_1 \\ \vdots \\ z_{n-1} \\ 0
\end{pmatrix}
\end{equation*}
(since the $\chi_j$ are holomorphic) and so 
we must have $-2\beta_jz_j+2\alpha_j z_j = -2\beta_jz_j-2\alpha_j z_j$,  hence $4\alpha_jz_j=0$ and $z_j=0$ for $j=1,\dots,n-1$.
 \end{proof}
 
Note that the argument above also yields the following.
 \begin{Add}\label{A:no-C-line}
$\dual'(p)$ is not $\C$-linear on any complex line in $H_pS$.
\end{Add}

From dimension considerations Lemma \ref{L:J-J*} has the following consequence.

\begin{Cor}\label{C:J-J*}
The map $J-J^*\colon H_pS\to H_pS$ is surjective.
\end{Cor}

\begin{proof}[Proof of Proposition \ref{P:d'-d''}]

Let $H_p^*S$ denote the real dual of $H_pS$. We claim that for any $\omega$ in  $H_p^*S
\otimes\C$  there are unique $\omega_1,\omega_2\in H_p^*S$ so that $\omega$ is the sum of the $J$-linear
$\omega'\eqdef\omega_1-i\omega_1\circ J$ and the $J^*$-linear $\omega''\eqdef\omega_2-i\omega_2\circ J^*.$  We then set $d'u=\left(du\rest_H\right)'$, $d''u=\left(du\rest_H\right)''$.

To prove the claim we show that the map 
\begin{align*}
\mapdef{H_p^*S\times H_p^*S}{H_p^* S\otimes \C}{\left(\omega_1,\omega_2\right)}{\left(\omega_1+\omega_2\right)
+i\left(-\omega_1\circ J-\omega_2\circ J^*\right)}
\end{align*}
is bijective. By dimension considerations it suffices to show that the map is injective.
But a pair $(\omega_1,\omega_2)$ in the kernel must satisfy $\omega_1=-\omega_2$, $\omega_1\circ J=-\omega_2\circ J^* = \omega_1\circ J^*$, hence $\omega_1\circ(J-J^*)=0$.
From Corollary \ref{C:J-J*} we now conclude that $\omega_1=0=\omega_2$.
\end{proof}

We note for future reference that $u$ is CR if and only if $du\rest_H=d'u$; this is equivalent in turn to the condition $d''u=0$.   Similarly, $u$ is dual CR if and only if $d'u=0$.

\begin{Remark}\label{A:H-decompe}
Later on we will make use of a corresponding decomposition of $H_pS\otimes\C$:
we claim that any vector $V\in H_pS\otimes \C$ decomposes uniquely as $V'+V''$, where $V'=V_1+iJ^*V_1$ and $V''=V_2+iJV_2$ for real $V_1,V_2\in H_pS$; equivalently, we claim that the map
\begin{align*}
\mapdef{H_pS\times H_pS}{H_p S\otimes \C}{\left(V_1,V_2\right)}{\left(V_1+V_2\right)
+i\left(J^*V_1+JV_2\right)}
\end{align*}
is bijective.  By dimension conslderations it suffices to show that the map is injective.  But a pair $
(V_1,V_2)$ in the kernel must satisfy $V_1=-V_2, J^*V_1=-JV_2=JV_1$, forcing $V_1=0=V_2$ by Lemma \ref{L:J-J*}.

Setting
\begin{align*}
H_p'S &= \left\{V+iJ^*V\colon V\in H_pS\right\}\\
H_p''S &= \left\{V+iJV\colon V\in H_pS\right\}
\end{align*}
we have shown that
\begin{equation*}
H_pS\otimes \C = H_p'S \oplus H_p''S.
\end{equation*}

A complex vector field on $S$ with values in $HS\otimes \C$ is in fact $H'$-valued if and only if it is annihilated by $J^*$-linear 1-forms; it is $H''$-valued if and only if it is annihilated by $J$-linear 1-forms
$\lozenge$
\end{Remark}

Moving forward, we will also need linear operators $\dpt$ and $\dppt$ mapping functions on $S$ to 1-forms on $S$ and satisfying $\dpt u\rest_H = d'u, \dppt u\rest_H = d''u$.  The operators $\dpt$ and $\dppt$ are not uniquely determined but explicit specific choices of such operators are offered in \eqref{E:dpt-def} and \eqref{E:dpt-def-hd} below and these also yield explicit formulas for $d'$ and $d''$.

Let $u$ be a smooth function on $S$.  We pose the question of whether $u$ may be decomposed as the sum of a CR function and a dual CR function.  This problem -- previously examined in [BG] -- is a natural analogue of the classical problem of characterizing functions decomposable (at least locally) as the sum of a CR function and a conjugate-CR function (equivalently, of characterizing traces of pluriharmonic functions).  Results on the latter problem are reviewed in \S\ref{S:plh} below.

\begin{Thm}\label{T:alt-Lee}
If $S$ is strongly $\C$-convex and simply-connected and $\theta$ is a $\C$-valued contact form on $S$ (that is, a non-vanishing complex 1-form with $\theta\rest_H\equiv0$) then the following conditions on smooth
$u\colon S\to\C$ 
are equivalent:
\refstepcounter{equation}\label{N:Lee-n=2-remix}
\begin{enum}
\item  $u$ decomposes as a sum $f+g$ where $f$ is CR and $g$ is dual-CR; \label{I:decomp}
\item there is a scalar function $\lam$ so that $\dpt u +\lam\theta$ is closed. \label{I:lam-exists}
\end{enum}
\end{Thm}

Note that   if $S$ fails to be simply-connected the result will still hold locally.

An explicit choice of contact form $\theta$ is offered in \eqref{E:the-1-forms} (and again in \eqref{E:the-1-formsn}) below.  (That choice is not $\R$-valued.)

Theorem \ref{T:alt-Lee} and its proof are adapted from [Lee, Lemma 3.1].

\begin{proof} If $u$ decomposes as a sum $f+g$ where $f$ is CR and $g$ is dual-CR then 
from the discussion following the proof of Proposition \ref{P:d'-d''} we have $d'u\rest_H=d'f\rest_H=df\rest_H$, hence $\dpt u=df-\lam\theta$ for smooth scalar $\lam$ and so 
\itemref{N:Lee-n=2-remix}{I:lam-exists}
 holds.

Conversely, if \itemref{N:Lee-n=2-remix}{I:lam-exists} holds we may write $\dpt u + \lam\theta=df$; it follows that $d'u=df\rest_H=d'f\rest_H$ and hence that $f$ is CR and that $g\eqdef u-f$ satisfies $d'g=0$ and thus is dual CR.
\end{proof}

  Note that \itemref{N:Lee-n=2-remix}{I:lam-exists} implies that
  \begin{equation}\label{E:nec-cond}
d\dpt u\rest_H=-\lam \,d\theta\rest_H.
\end{equation}
Note also that the strong pseudoconvexity of $S$ guarantees that $d\theta\rest_H$ is nowhere-vanishing.

\medskip

\boxed{\mathbf{n=2}}
We may define
\begin{equation}\label{E:lam-def}
\lam=-\frac{d\tilde{d'}u\rest_H}{d\theta\rest_H},
\end{equation}
then check whether or not this works.  See Theorem A
%\ref{T:PDE-for-decomp} 
for the result of this approach.

\medskip

\boxed{\mathbf{n>2}}
In higher dimension we have the following result.
\begin{Thm}\label{T:equiv-hd}
For $S$, $\theta$ and $u$ as in Theorem \ref{T:alt-Lee}  the following are equivalent:
\refstepcounter{equation}\label{N:Lee-hd}
\begin{enum}
\item there is a smooth scalar function $\lam$ so that \eqref{E:nec-cond} holds;\label{I:nec}
\item $u$ may be decomposed as the sum of a CR function and a dual CR function.\label{I:hd-decomp}
\end{enum}
\end{Thm}

\begin{proof}
The discussion above shows that \itemref{N:Lee-hd}{I:hd-decomp} implies \itemref{N:Lee-hd}{I:nec}.

Suppose on the other hand that the necessary condition \itemref{N:Lee-hd}{I:nec} holds.  

Then the restriction of $d\left(\dpt u + \lam\theta\right)= d\dpt u+\lam \,d\theta + d\lam\w \theta$ to $H$ vanishes identically.  
In view of the dimension condition and the non-degeneracy of $H$,  Lemma 3.2 from [Lee] tells us that a closed 2-form whose restriction to $H$ vanishes must vanish identically, hence in particular
$d\left(\dpt u + \lam\theta\right)=0$ and thus \itemref{N:Lee-n=2-remix}{I:lam-exists} holds.  Theorem \ref{T:alt-Lee} now furnishes the desired decomposition.
\end{proof}

\bigskip

\section{The projective decomposition problem for $n=2$}\label{S:n=2-projiharm}

We make the standing assumption that $S$ is a strongly $\C$-convex hypersurface satisfying \eqref{E:no-hit-0}.  (Note that \eqref{E:no-hit-0} holds automatically if $S$ is a compact hypersurface enclosing $0$ [APS, \S2.5].)

We define $w_1(z)$ and $w_2(z)$ as in \eqref{E:wj}. 

\begin{Lem}\label{L:Leg} We have
\begin{subequations}\label{E:Leg}\begin{align}
z_1w_1+z_2w_2&= 1 \text{ on }S\label{E:Leg-a}\\
w_1\,dz_1+w_2\,dz_2+z_1\,dw_1+z_2\,dw_2 &= 0\text{ as 1-forms on }S\label{E:Leg-b}\\
w_1\,dz_1+w_2\,dz_2 &= 0 \text{ as forms on  }H\label{E:Leg-c}\\
z_1\,dw_1+z_2\,dw_2 &= 0 \text{  as forms on }H\label{E:Leg-d}
\end{align}
\end{subequations}
\end{Lem}

\begin{proof} Equation \eqref{E:Leg-a} is immediate; \eqref{E:Leg-b} follows from differentiation of \eqref{E:Leg-a}. 
Equation \eqref{E:Leg-c} follows from the fact that $\dee\rho$ vanishes along $H_pS$; then \eqref{E:Leg-d} follows by combining \eqref{E:Leg-b} with \eqref{E:Leg-c}.
\end{proof}

\begin{Lem}\label{L:tot-real}
At each point of $S$ at least one of
\begin{align*}
&dz_1\w dz_2\w dw_1\\
&dz_1\w dz_2\w dw_2
\end{align*}
is non-zero as a 3-form on $S$
\end{Lem}

\begin{proof}
From Addendum \ref{A:no-C-line} we see that at least one of the $dw_j$ fails to be $\C$-linear; the claim follows immediately.
\end{proof}

\begin{Lem}\label{L:cap-line}
The intersection of $S$ with $\{z_j=0\}$  has no relative interior.\end{Lem}

\begin{proof}
This follows from the strong pseudoconvexity of $S$.
\end{proof}

\begin{Prop} \label{P:XT-def}
There are uniquely-defined tangential vector fields $X,T$ on $S$ satisfying 
\begin{align} %\label{E:XT-rules}
Xz_1&=%\bar Yz_1 =
0&
Xz_2&=%\bar Yz_2 =
0\notag\\
Xw_1&=z_2&
Xw_2&=-z_1\notag\\
%\bar Yw_1&=\bar\xi z_2&
%\bar Yw_2&=-\bar\xi z_1\notag\\
%X\bar{w}_1 &= \phi \bar{z}_2 & X \bar{w}_2&=-\phi \bar{z}_1\notag\\
%X\bar z_1 &= \alpha \bar w_2&
%X\bar z_2 &= -\alpha \bar w_1\notag\\%needed?
Tz_1&=w_2&
Tz_2 &= -w_1\label{E:XT-rules}\\
%\bar V z_1 &= \bar\sigma w_2&
%\bar V z_2 &= -\bar\sigma w_1\notag\\
%T \bar{z}_1 &= \psi \bar{w}_2 & T\bar{z}_2 &= -\psi \bar{w}_1 \notag\\
Tw_1&=%\bar Vw_1=
0&
Tw_2&=%\bar Vw_2=
0.\notag%\label{E:diff-rules}\\
\end{align}

$T$ and $X$ take values in $H'$ and $H''$, respectively.
\end{Prop}

\begin{proof}
Consider $p\in S$ and suppose that $dz_1\w dz_2\w dw_1\ne 0$ at $p$; then linear independence guarantees the existence and uniqueness of $X$ and $T$ in a neighborhood of $p$ satisfying all conditions above other than $Xw_2=-z_1$ and $Tw_2=0$.  If $z_2(p)\ne 0$ then the remaining equations follow from differentiation of $z_1w_1+z_2w_2=1$; if $z_2(p)= 0$ then by Lemma \ref{L:cap-line} the remaining equations still must hold at many points near to $p$, hence by passing to the limit they must also hold at $p$.

A similar argument holds if $dz_1\w dz_2\w dw_2\ne 0$ at $p$; uniqueness guarantees that the local solutions patch together to form a global solution.
\end{proof}

Let $\Rt = [X,T]$. (This corresponds to $iR$ in the notation from [BG].) %to avoid mishaps citing material from [BG] 

\begin{Prop} \label{P:Rt-rules}  We have
\begin{align}
\Rt z_1& =-z_1 & \Rt z_2& =-z_2\notag\\
\Rt w_1&=w_1 & \Rt w_2&=w_2.\label{E:Rt-rules} 
\end{align}
\end{Prop}

\begin{proof}
These follow directly from \eqref{E:XT-rules}.
\end{proof}

\begin{Prop}\label{P:2-brack}
We have \begin{align*}
[\Rt, X] &= -2X\\
[\Rt, T] &= 2T.
\end{align*}
\end{Prop}

\begin{proof}
By Lemma \ref{L:tot-real} it suffices to use \eqref{E:XT-rules} and \eqref{E:Rt-rules}  to test both sides against $z_j,w_j$.
\end{proof}

Now let
\begin{align}\label{E:the-1-forms}
\eta' &= z_2\,dz_1-z_1\,dz_2 \notag\\
\eta'' &= w_2\,dw_1-w_1\,dw_2 \\
\theta &= -w_1\,dz_1-w_2\,dz_2= z_1\,dw_1+z_2\,dw_2.\notag 
\end{align}
(The equivalence of the two descriptions of $\theta$ follows from  \eqref{E:Leg-b}.)

\begin{Lem}\label{L:the-1-forms}
We have
\begin{equation}\label{E:d-remix}
du = (Tu)\,\eta' + (Xu)\,\eta'' + (\Rt u)\,\theta. 
\end{equation}
\end{Lem}

\begin{proof}
Direct computation using  \eqref{E:the-1-forms} reveals that 
\begin{align}\label{E:the-1-forms-reversed}
dz_1 &=  w_2\eta'-z_1\theta\notag\\
dz_2 &= -w_1\,\eta'-z_2\theta\notag\\
dw_1 &= z_2\eta''+w_1\theta\\
dw_2 &= -z_1\eta''+w_2\theta\notag;
\end{align}
using \eqref{E:XT-rules} it follows that  that \eqref{E:d-remix} holds for $u=z_1,z_2,w_1$ or $w_2$.

From Lemma \ref{L:tot-real} we see that this implies the general case.
\end{proof}

\begin{Lem}\label{L:pairing-rulz} We have
\begin{align}\label{E:pairing-rulz}
\theta(X)&=0&\eta'(X)&=0 &\eta''(X)&=1\notag\\
\theta(T)&=0&\eta'(T)&=1&\eta''(T)&=0\\
\theta(\Rt )&=1&\eta'(\Rt )&=0&\eta''(\Rt )&=0.\notag
\end{align}
\end{Lem}

Note that from $\theta(X)=0=\theta(T)$ we see that 
$\theta$ is a contact form on $S$; as we quote below from \S \ref{S:d'd''} above we will use this choice of contact form.  Note also that \eqref{E:pairing-rulz} shows that $\eta', \eta''$ and $\theta$ are linearly independent at each point of $S$.

\begin{proof}
From \eqref{E:the-1-forms} and \eqref{E:XT-rules}   we have
\begin{align*}
\theta(X) &= -w_1(Xz_1)-w_2(Xz_2)=0\\
\theta(T) &=z_1(Tw_1)-z_2(Tw_2)=0\\
\theta(\Rt ) &= -w_1(\Rt  z_1)-w_2(\Rt  z_2)=1.
\end{align*}
Similar computations serve to verify the remaining entries.
\end{proof}

From \eqref{E:the-1-forms} and \eqref{E:the-1-forms-reversed} we find that
\begin{align}\label{E:d-of-the-1-forms}
d\eta' &= -2\,dz_1\w dz_2= 2\,\eta'\w\theta\notag\\
d\eta'' &= -2\,dw_1\w dw_2= -2\,\eta''\w\theta \\
d\theta &= dz_1\w dw_1+dz_2\w dw_2 = \eta'\w\eta''.\notag
\end{align}

Returning to the discussion from \S \ref{S:d'd''} we now set
\begin{align}\label{E:dpt-def}
\dpt u&=(Tu)\,\eta'\notag\\
\dppt u&=(Xu)\,\eta''\\
d^0 u&=(\Rt u)\,\theta\notag
\end{align}
so that
\begin{equation*}
d=\dpt+\dppt+d^0
\end{equation*}
and
\begin{align*}
\dpt u\rest_H&= d'u\\
\dppt u\rest_H&= d''u\\
d^0 u\rest_H &= 0.
\end{align*}

\bigskip

In Proposition \ref{P:Lee-n=2-remix} and Theorem A
%\ref{T:PDE-for-decomp} 
below, the  strongly $\C$-convex hypersurface $S\subset\C^2$ is assumed to be simply-connected.  (In the case of compact $S$, the simple-connectivity holds automatically since $S$ will be diffeomorphic to the sphere $S^3$ -- one way to show this is to extend the result in [Sem, \S5] using the results of [Lem].)

\begin{Prop}\label{P:Lee-n=2-remix}
If the equivalent conditions of Theorem \ref{T:alt-Lee} hold (with the above choice of $\theta$) then $\lam=\Rt f= XTf = XTu$. 
\end{Prop}

\begin{proof}
From the proof of Theorem \ref{T:alt-Lee} we have
\begin{align*}
\dpt u + \lam\theta &= \dpt f + \dppt f + d^0 f\\
&= \dpt f + \left(\Rt f\right)\theta;
\end{align*}
matching terms we find that $\lam=\Rt f= XTf = XTu$ as claimed.
\end{proof}

We need to better understand the condition that
$\tilde{d'u} + (XTu)\,\theta=(Tu)\,\eta'+ (XTu)\,\theta$
is closed, i.e., 
\begin{align*}%\label{E:key-cond-remix}
0&=d(Tu)\w\eta' + (Tu)\cdot d\eta'  + d(XTu)\w\theta+(XTu)\cdot d\theta\\
&= \left( (TTu)\,\eta' +(XTu)\,\eta''+(\Rt Tu)\,\theta\right)\w\eta'
+2(Tu)\,\eta'\w\theta\\
&\qquad +\left( (TXTu)\,\eta' +(XXTu)\,\eta''+(\Rt XTu)\,\theta\right)\w\theta
-(XTu)\,\eta''\w\eta'\\
&= \left(  -\Rt Tu+2Tu+TXTu\right)\,\eta'\w\theta+\left(XXTu\right)\,\eta''\w\theta\\
&= \left(  2Tu-[\Rt,T]u+TTXu\right)\,\eta'\w\theta+\left(XXTu\right)\,\eta''\w\theta\\
&= \left(  TTXu\right)\,\eta'\w\theta+\left(XXTu\right)\,\eta''\w\theta.
\end{align*}

We have proved the following.

\begin{MThm-A}
\label{T:PDE-for-decomp}
For $S\subset\C^n \: (n=2)$ strongly $\C$-convex and simply-connected, the following conditions on smooth $u\colon S\to\C$ are equivalent:
\refstepcounter{equation}\label{N:PDE-for-decomp}
\begin{enum}
\item $u$ decomposes as a sum $f+g$ where $f$ is CR and $g$ is dual-CR; 
\item $XXTu=0=TTXu$. \end{enum}
\end{MThm-A}

\section{Alternate construction of $X$, $T$ and $\Rt$} \label{S:alt}

In this section we set out an alternate approach to the development of the vector fields $X,T,\Rt$.

Let
$\inc'=\{(z_1,z_2,w_1,w_2)\in\C^4 \,|\, z_1w_1+z_2 w_2 = 1\}.$   (See Remark \ref{R:inc} below.)

The holomorphic vector fields 
\begin{align*}
\Xb&=z_2\,\frac{\dee}{\dee w_1}-z_1\,\frac{\dee}{\dee w_2}\\
\Tb&=w_2\,\frac{\dee}{\dee z_1}-w_1\,\frac{\dee}{\dee z_2}\\
\Rtb&=-z_1\,\frac{\dee}{\dee z_1}-z_2\,\frac{\dee}{\dee z_2}+w_1\,\frac{\dee}{\dee w_1}+w_2\,\frac{\dee}{\dee w_2}.
\end{align*}
on $\C^4$ are tangent to $\inc'$.  We have
\begin{align}
[\Xb,\Tb] &= \Rtb\notag\\
[\Rtb,\Xb] &= -2\Xb \label{E:big-brack}\\
[\Rtb,\Tb] &= 2\Tb.\notag
\end{align}

Consider the diffeomorphism
\begin{align*}
\mapdef
{\dual^\sharp_S\st S} {\Gamma_S\subset \inc'}{\left(z_1,z_2\right)}{\left(z_1,z_2,w_1(z),w_2(z)\right)}.
\end{align*}
with $\Gamma_S\eqdef \dual_S^\sharp(S)$. Using Addendum \ref{A:no-C-line}
we see that $\Gamma_S$ is a totally real 3-manifold inside the complex 3-manifold $\inc'$.

\begin{Prop}\label{P:push}
We have
\begin{subequations}\label{E:push}
\begin{align}
\left(\dual^\sharp_S\right)_*X&=\left(\Xb+\phi\bar\Xb+\alpha\bar\Tb\right)\rest_{\Gamma_S}\label{E:pushX}\\
\left(\dual^\sharp_S\right)_*T&=\left(\Tb+\beta\bar\Xb+\psi\bar\Tb\right)\rest_{\Gamma_S}
\label{E:pushT}\\
\left(\dual^\sharp_S\right)_*\Rt&=\left(\Rtb + (X\beta-T\phi)\bar\Xb +(X\psi-T\alpha)\bar\Tb+(\phi\psi-\alpha\beta)\bar\Rtb\right)\rest_{\Gamma_S}\label{E:pushRt}
\end{align}
\end{subequations}
for certain smooth functions $\alpha$, $\beta$, $\psi$, $\phi$.
\end{Prop}

\begin{proof}
Quoting (4.1)  from [BG] (but correcting typos in the last two entries) there are smooth functions 
$\alpha$, $\beta$, $\psi$, $\phi$ satisfying
\begin{align*}
X\bar z_1 &= \alpha \bar w_2 & X\bar z_2 &= -\alpha \bar w_1\notag\\
X\bar{w}_1 &= \phi \bar{z}_2 & X \bar{w}_2&=-\phi \bar{z}_1\notag\\
T \bar{z}_1 &= \psi \bar{w}_2 & T\bar{z}_2 &= -\psi \bar{w}_1 \notag\\
 T\bar w_1 &= \beta \bar z_2 & T\bar w_2 &= -\beta \bar z_1.\notag
%
%Xw_1&=z_2&
%Xw_2&=-z_1\notag\\
%\bar Yw_1&=\bar\xi z_2&
%\bar Yw_2&=-\bar\xi z_1\notag\\
%Xz_1&=\bar Yz_1=0&
%Xz_2&=\bar Yz_2=0\notag\\
%Tz_1&=w_2&
%Tz_2 &= -w_1\label{E:diff-rules}\\
%\bar V z_1 &= \bar\sigma w_2&
%\bar V z_2 &= -\bar\sigma w_1\notag\\
%T \bar{z}_1 &= \psi \bar{w}_2 & T\bar{z}_2 &= -\psi \bar{w}_1 \notag\\
%Tw_1&=\bar Vw_1=0&
%Tw_2&=\bar Vw_2=0\notag\\
\end{align*}
The first two lines of \eqref{E:push} follow from applying both sides to the functions $z_j,\bar z_j, w_j, \bar{w}_j$:  for \eqref{E:pushX} application of either side to $z_1,z_2,\bar z_1,\bar z_2,w_1,w_2,\bar w_1,\bar w_2$ leads to  $0,0,\alpha\bar w_2,-\alpha\bar w_1,w_1,w_2,\phi \bar z_2,-\phi\bar z_1$, respectively, while a similar computation verifies \eqref{E:pushT}. 
The remaining line \eqref{E:pushRt} now follows from a bracket computation using the previous results.
\end{proof}

Any vector field $V$ (with values in $T\inc'$) defined on $\Gamma_S$ may be written uniquely as $V^{\sf tang} + V^{\sf normal}$, where $V^{\sf tang}$ and $JV^{\sf normal}$ are tangent to $\Gamma_S$. 

\begin{Prop} \label{P:tang}
 We have
\begin{align*}
\left(\dual^\sharp_S\right)_*X &= 2 \left(  \Xb\rest_{\Gamma_S} \right)^{\sf tang}\\
\left(\dual^\sharp_S\right)_*T &= 2 \left(  \Tb \rest_{\Gamma_S}\right)^{\sf tang}\\
\left(\dual^\sharp_S\right)_*\Rt &= 2 \left(  \Rtb \rest_{\Gamma_S}\right)^{\sf tang}.
\end{align*}
\end{Prop}

Thus $X$ is the unique vector field on $S$ pushing forward to twice the tangential part of $\Xb$
-- that is, 
$ X = 2\left(\dual^\sharp_S\right)\inv_*
\left( \left(  \Xb\rest_{\Gamma_S} \right)^{\sf tang}\right)
$ -- 
and similarly for $T$ and $\Rt$.

We will prove Proposition \ref{P:tang} as a consequence of a related result using type considerations. Recall that any vector field $V$ on a subset of $\inc'$ decomposes uniquely as $V^{(1,0)}+V^{(0,1)}$ with
\begin{align*}
V^{(1,0)}&=\frac12\left( V-iJV \right) \\
V^{(0,1)}&= \frac12\left( V+iJV \right). 
\end{align*}
Holomorphic vector fields are of type (1,0).  % and thus must satisfy $JV=iV$.

\begin{Prop}\label{P:type}
We have
\begin{align*}
\left( \left(\dual^\sharp_S\right)_*X \right)^{(1,0)} &= \Xb\rest_{\Gamma_S}\\
\left( \left(\dual^\sharp_S\right)_*T \right)^{(1,0)} &= \Tb\rest_{\Gamma_S}\\
\left( \left(\dual^\sharp_S\right)_*\Rt \right)^{(1,0)} &= \Rtb\rest_{\Gamma_S}.
\end{align*}
\end{Prop}

\begin{proof}
These follow directly from Proposition \ref{P:push}.
\end{proof}

\begin{Lem}\label{L:tang-type}
If $V$ is a vector field tangent to $\Gamma_S$ then $V=2\left( V^{(1,0)} \right)^{\sf tang}$.
\end{Lem}

\begin{proof}
This follows from $2V^{(1,0)}=V-iJV$.
\end{proof}

\begin{proof}[Proof of Proposition \ref{P:tang}]
Apply Lemma \ref{L:tang-type} to the results of Proposition \ref{P:type}.
\end{proof}

\begin{Remark}\label{R:inc}
We may identify $\inc'$ with an open subset of the full incidence manifold 
\[\inc \eqdef \{\left((z_0:z_1:z_2),(w_0:w_1:w_2)\right)\in\CP^2\times \CP^2\,|\, z_0w_0+z_1w_1+z_2w_2=0\}\] 
important in projective duality theory (as discussed in   [APS, \S3.2]) via the map
\begin{align*}
 \inc'  &\rightarrow \inc\setminus \{z_0w_0=0\} 
\\
 (z_1,z_2,w_1,w_2)   &\mapsto \left((i:z_1:z_1),(i:w_1:w_1)\right).
\end{align*}.
\end{Remark}

We may also
identify $\inc'$ with $SL(2,\C)$ via 
$(z_1,z_2,w_1,w_2)\mapsto
\begin{pmatrix}
z_1  &  -w_2    \\
 z_2 &  w_1    
\end{pmatrix}.$  Then the flows $\exp(t\Re \Xb)$, $\exp(t\Re \Tb)$ and $\exp(t\Re \Rtb)$
correspond to  right-multiplication by $\begin{pmatrix}
1  &  t   \\
0 &  1    
\end{pmatrix}$,
$\begin{pmatrix}
1  &  0   \\
-t &  1    
\end{pmatrix}$  and
$\begin{pmatrix}
e^{-t}  &  0   \\
0&  e^t    
\end{pmatrix}$
respectively.

\section{The projective decomposition problem for $n>2$}\label{S:n>2-projiharm}

Again we make the standing assumption that $S$ is a strongly $\C$-convex hypersurface satisfying \eqref{E:no-hit-0}.

We define $w_k(z)$ as in \eqref{E:wj}. 

\begin{Lem}\label{L:Legn} We have
\begin{subequations}\label{E:Legn}\begin{align}
z_1w_1+\cdots+z_n w_n&= 1 \text{ on }S\label{E:Legn-a}\\
w_1\,dz_1+\cdots +w_n\,dz_n+z_1\,dw_1+\cdots+z_n\,dw_n &= 0\text{ as 1-forms on }S\label{E:Legn-b}\\
w_1\,dz_1+\cdots +w_n\,dz_n &= 0 \text{ as forms on  }H\label{E:Legn-c}\\
z_1\,dw_1+\cdots+z_n\,dw_n &= 0 \text{  as forms on }H.\label{E:Legn-d}
\end{align}
\end{subequations}
\end{Lem}

\begin{proof}
Like Lemma \ref{L:Leg}.
\end{proof}

\begin{Prop} \label{P:XT-def-higher-dim}
For $1\leq j,k \leq n$, $j \neq k$, there are uniquely-determined tangential vector fields $X_{jk},T_{jk}$ on $S$ satisfying 
\begin{align} %\label{E:XT-rules}
X_{jk}z_\ell&=0& T_{jk}w_\ell&=0 \notag\\
X_{jk}w_\ell &= \begin{cases} z_k & \ell=j \\
-z_j  & \ell=k\\ 
0 & otherwise
\end{cases}& T_{jk}z_\ell &= \begin{cases} w_k & \ell=j \\
-w_j  & \ell=k\\ 
0 & otherwise. 
\end{cases}
\end{align}
The $T_{jk}$ and $X_{jk}$ take values in $H'$ and $H''$, respectively.
\end{Prop}

For $k=j$, we  set $X_{jj}=0=T_{jj}$.

Note the relation
\begin{equation}\label{E:X-relation}
z_j X_{k\ell}+z_k X_{\ell j}+z_\ell X_{jk}=0.
\end{equation}

\begin{Prop} \label{P:Rtn-rules} There is a uniquely-determined tangential vector field $\Rt$ on $S$ satisfying
\begin{alignat}{3}
\Rt z_1& =-z_1\qquad    &\cdots & \qquad\Rt z_n& =-z_n \notag\\
\Rt w_1&=w_1   &\cdots& \qquad \Rt w_n&=w_n.\label{E:Rt-rules-hd} 
\end{alignat}
\end{Prop}

\begin{proof}[Proofs of Propositions \ref{P:XT-def-higher-dim} and \ref{P:Rtn-rules}.]
These are similar to the proof of Proposition \ref{P:XT-def}.
\end{proof}

\begin{Prop}\label{P:2-brackn}
We have \begin{align*}
[\Rt, X_{jk}] &= -2X_{jk}\\
[\Rt, T_{jk}] &= 2T_{jk}.
\end{align*}
\end{Prop}

\begin{proof}
Similar to the proof of Proposition \ref{P:2-brack}.
\end{proof}

\begin{Prop}\label{P:span}
For $p\in S$ we have 
\begin{align*}
H_p' &= \Span\left\{T_{jk}\colon 1\leq j,k \leq n\right\}\\
H_p'' &= \Span\left\{X_{jk}\colon 1\leq j,k \leq n\right\}.
\end{align*}
\end{Prop}

\begin{proof}  First note that $\dim_\C H_p' = n-1 = \dim_\C H_p''$; then note that
after possibly reordering the coordinates we may assume that $z_1w_1\ne0$ at $p$  and thus that $T_{12},\dots,T_{1n}$ are linearly independent in $H_p'$ while $X_{12},\dots,X_{1n}$ are linearly independent in $H_p''$.  
\end{proof}

Now let
\begin{align}\label{E:the-1-formsn}
\eta'_{jk}&= z_k\,dz_j-z_j \,dz_k \notag\\
\eta''_{jk} &= w_k\,dw_j-w_j \,dw_k\\
\theta &= -w_1\, dz_1-\cdots-w_n\, dz_n= z_1\, dw_1+\cdots+z_n\,dw_n.\notag 
\end{align}

Note that
\begin{equation}\label{E:d-theta}
d\theta=dz_1\w dw_1+\cdots+dz_n\w dw_n
\end{equation}
and that $d\theta\rest_H$ is non-degenerate.

\begin{Lem}\label{L:the-1-forms-hd}
We have
\begin{equation}\label{E:d-remix-hd}
du = \frac12 \sum_{j,k} \left(T_{jk}u\right)\,\eta'_{jk}+\frac12\sum_{j,k} \left(X_{jk}u\right)\,\eta''_{jk} +\left(\Rt u\right)\,\theta. 
\end{equation}
\end{Lem}

\begin{proof}
Check that the result holds for $u=z_j$ or $w_j$, then apply adapted version of 
Lemma \ref{L:tot-real}.  
\end{proof}

Following \eqref{E:dpt-def} we set
\begin{align}\label{E:dpt-def-hd}
\dpt u&=\frac12 \sum_{j,k} \left(T_{jk}u\right)\,\eta'_{jk}\notag\\
\dppt u&=\frac12 \sum_{j,k} \left(X_{jk}u\right)\,\eta''_{jk}\\
d^0 u&=(\Rt u)\,\theta\notag
\end{align}
so that again
\begin{equation*}
d=\dpt+\dppt+d^0.
\end{equation*}

Let $\omega$ be an $H$-form of degree two.  We will say that  $\omega$ has
\begin{itemize}
\item  type (2,0) if $\omega$ is $J$-bilinear;
\item  type (0,2) if $\omega$ is $J^*$-bilinear;
\item  type (1,1) if it can be written as a finite sum of wedge products of $J$-linear $H$-forms of degree one with $J^*$-linear $H$-forms of degree one.
\end{itemize}

From standard arguments we obtain the following.

\begin{Prop}\label{P:2-decomp}
Every $H$-form $\omega$ of degree two decomposes uniquely as a sum $\omega^{(2,0)}+\omega^{(0,2)}+\omega^{(1,1)}$ of forms of specified type.
\end{Prop}

Recall the decomposition of $H_pS\otimes\C$ from Remark \ref{A:H-decompe}.

\begin{Lem}\label{L:(1,1)-decomp}
If an $H$-form $\omega$ has type (1,1) then $\omega\left(V'+V'',W'+W''\right)=\omega\left(V',W''\right)-\omega\left(W',V''\right)$.
\end{Lem}

\begin{proof}
Using the definition of a form of type (1,1) and  the last sentence of Remark \ref{A:H-decompe} we have $\omega\left(V',W'\right)=0=\omega\left(V'',W''\right)$; the claim follows.
\end{proof}

\begin{Prop}\label{P:num-type}
$d\dpt u\rest_H=\frac14 \sum\limits_{j,k,\ell,m} \left(X_{\ell m}T_{jk}u\right)\eta''_{\ell m}\w\eta'_{jk}$; in particular, $d\dpt u\rest_H$ has type (1,1).
\end{Prop}

\begin{proof} We first show that $d\dpt u\rest_H$ has type (1,1).

Note first that $d\eta'_{jk}\rest_H$ has type (2,0), $d\eta''_{jk}\rest_H$ has type (0,2)
and $d\theta\rest_H$ has type (1,1).

From direct inspection we now find that 
\begin{align*}
\left( d\dpt u\rest_H\right)^{(0,2)}&=0\\
\left( d\dppt u\rest_H\right)^{(2,0)}&=0\\
\left( d d^0 u\rest_H\right)^{(2,0)}&=0.
\end{align*}

It suffices now to show that $\left( d\dpt u\rest_H\right)^{(2,0)}=0$; this follows from
taking (2,0)-components in
\begin{align*}
0 &= ddu\rest_H\\
&= d\dpt u\rest_H + d\dppt u\rest_H + dd^0 u\rest_H.
\end{align*}

Ignoring the cancelling (2,0)-terms 
now find that
\begin{align*}
d\dpt u\rest_H &= \frac12 \sum_{j,k} \dppt\left(T_{jk}u\right)\,\eta'_{jk}\rest_H\\
&= \frac14 \sum\limits_{j,k,\ell,m} \left(X_{\ell m}T_{jk}u\right)\eta''_{\ell m}\w\eta'_{jk}\rest_H.
\end{align*}
\end{proof}

For conciseness we now fix $p\in S$ and set 
\begin{equation}\label{E:nu_p-def}
\nu_p=d\dpt u(p)\rest_H.
\end{equation}

\begin{Lem}\label{L:lam}
The condition \eqref{E:nec-cond} holds (at $p$) if and only if there is a scalar $\lam$ satisfying 
\begin{equation*}
\nu_p\left(T_{jk},X_{\ell m}\right)=\lam \cdot d\theta\left(T_{jk},X_{\ell m}\right)
\end{equation*}
for all $j,k,\ell,m$.
\end{Lem}

\begin{proof} This follows from Proposition \ref{P:num-type} along with Lemma \ref{L:(1,1)-decomp} and Proposition \ref{P:span}.
\end{proof}

\begin{Prop}\label{P:hess}  Suppose that
\begin{itemize}
\item $T$ is a vector field taking values in $H'$;
\item $X$ is a vector field taking values in $H''$;
\item $d\theta(T,X)\equiv0$.
\end{itemize}
Then $\nu_p\left(T,X\right)=-XTu(p)$.
\end{Prop}

\begin{Lem}  \label{L:Taylor}
We can write
\begin{equation}\label{E:u-decomp}
u=C+f_1+g_2+\sum\limits_{j=3}^N f_j g_j + E
\end{equation}
with 
\begin{itemize}
\item all $f_j$ are CR;
\item all $f_j(p)=0$;
\item all $g_j$ are dual CR;
\item all $g_j(p)=0$;
\item all second derivatives of $E$ vanish at $p$.
\end{itemize}
\end{Lem}

\begin{proof}
The computations from the proof of Lemma \ref{L:J-J*} show that 
\begin{equation*}
\Span\left\{dz_j(p),dw_k(p)\right\}=T_pS\otimes \C.
\end{equation*}
A "Taylor polynomial"-type argument now serves to prove the Lemma.
\end{proof}

\begin{proof}[Proof of Proposition \ref{P:hess}]
Invoking the  decomposition from Lemma \ref{L:Taylor} we note first that 
\begin{equation}\label{E:d dpt C}
\dpt C = 0 =-XT(C)(p)
\end{equation}
and 
\begin{equation}\label{E:d dpt g}
d\dpt g_2=0=-XTg_2(p).
\end{equation}

Next we note that \[\dpt f_1 = df_1-d^0 f_j = df_1 - (\Rt f_1)\,\theta\]
(since $\dppt f_1=0$)  and thus 
\begin{equation}\label{E:d dpt f}
d\dpt f_1(T,X)
= (\Rt f_1)(p)\cdot d\theta(T,X)= 0
\end{equation}
(using $\theta(T)=0=\theta(X)$).

For the general term $f_jg_j$ we have
\begin{align*}
\dpt\left(f_jg_j\right)
&=\dpt\left(f_j\right) g_j\\
&=\left( df_j-d^0 f_j  \right)g_j\\
&=g_j \,df_j-g_j\left(\Rt f_j\right)\theta
\end{align*}
and so
\begin{align}
d\dpt\left(f_jg_j\right)
&= - df_j\w dg_j - d\left(g_j\left(\Rt f_j\right)\right)\w\theta
- g_j\left(\Rt f_j\right)\,d\theta \notag \\
d\dpt\left(f_jg_j\right)\rest_H
&= - df_j\w dg_j\rest_H - g_j\left(\Rt f_j\right)\,d\theta\rest_H \notag \\
d\dpt\left(f_jg_j\right)(T,X) &= - \left(df_j\w dg_j\right)(T,X)=-Tf_j\cdot Xg_j \notag \\
\intertext{thus}
d\dpt\left(f_jg_j\right)(T,X)(p) &= -Tf_j(p)\cdot Xg_j(p) = -XT(f_jg_j)(p).
\label{E:d dpt fg}
\end{align}

Adding \eqref{E:d dpt C}, \eqref{E:d dpt f}, \eqref{E:d dpt g} and \eqref{E:d dpt fg} and recalling 
\eqref{E:nu_p-def} and \eqref{E:u-decomp}
we obtain $\nu_p\left(T,X\right)=-XTu(p)$.
\end{proof}

Let 
\begin{align*}
\Xt_{jk\ell} &= w_j X_{\ell j} + w_k X_{\ell k}\\
\Tt_{jk\ell} &= z_j T_{\ell j} + z_k T_{\ell k}.
\end{align*}
The   $\Xt_{jk\ell}$ and $\Tt_{jk\ell}$ take values in $H''$ and $H'$, respectively with
\begin{align} %\label{E:XT-rules}
\Xt_{jk\ell}z_m&=0& \Tt_{jk\ell}w_m&=0 \notag\\
\Xt_{jk\ell}w_m &= \begin{cases} -z_jw_\ell & m=j \\
-z_kw_\ell & m=k\\
z_jw_j+z_kw_k & m=\ell\\ 
0 & otherwise
\end{cases}& \Tt_{jk\ell}z_m &= \begin{cases} -z_\ell w_j & m=j \\
-z_\ell w_k & m=k\\
z_jw_j+z_kw_k & m=\ell\\ 
0 & otherwise.
\end{cases}
\end{align}

\begin{Lem}\label{L:ortho}
For $j,k,\ell$ distinct we have $d\theta\left(\Tt_{jk\ell},X_{jk}\right)=0=d\theta\left(T_{jk},\Xt_{jk\ell}\right)$.
\end{Lem}

\begin{proof}  Using \eqref{E:d-theta} we find that 
$d\theta\left(\Tt_{jk\ell},X_{jk}\right)=-w_\ell z_j z_k+w_\ell z_j z_k=0$ and
$d\theta\left(T_{jk},\Xt_{jk\ell}\right)= -z_\ell w_j w_k+z_\ell w_j w_k=0 $.
\end{proof}

For $X\in H_p'S$ we set
\begin{align*}
X^{\perp_{d\theta}}&= \left\{ T\in  H_p''S\colon d\theta(T,X)(p)=0\right\}\\
X^{\perp_{\nu_p}}&= \left\{ T\in  H_p''S\colon \nu_p(T,X)=0\right\}.\\
\end{align*}

\begin{Lem}\label{L:tri-perp}  
\begin{enumerate}
\item If $z_j\ne 0$ the vectors $\left\{ X_{jk}(p)\colon k\ne j\right\}$ form a basis for $H_p''$.
\item If \eqref{E:star} holds and $k\ne j$ the vectors 
$\left\{ \Tt_{jk\ell}(p)\colon \ell\notin\{j,k\}\right\}$
form a basis of $X_{jk}^{\perp_{d\theta}}(p)$.
\end{enumerate}
\end{Lem}

\begin{proof}
In each case we have the right number of linearly independent vectors in the indicated space.
\end{proof}

The following result is based on Theorem 3 in [Aud].

\begin{MThm-B}
\label{T:Aud} For $S\subset\C^n \: (n>2)$ strongly $\C$-convex and simply-connected
and satisfying \eqref{E:star} the following conditions on smooth $u\colon S\to\C$ are equivalent.
\refstepcounter{equation}\label{N:Aud-remix}
\begin{enum}
\item  $u$ decomposes as a sum $f+g$ where $f$ is CR and $g$ is dual-CR; \label{I:decomp-hd}
\item for all distinct $j,k,\ell$ we have
\begin{equation}\label{E:Aud2}
X_{jk}\Tt_{jk\ell}u = 0. 
\end{equation}
\end{enum}
\end{MThm-B}

\begin{proof}
If \itemref{N:Aud-remix}{I:decomp-hd} holds then using \eqref{E:nec-cond} along with Lemma \ref{L:ortho} and Proposition \ref{P:hess} we have
\begin{align*}
X_{jk}\Tt_{jk\ell}u(p) &= \nu_p\left( \Tt_{jk\ell}, X_{jk} \right)\\
&= \lam\cdot d\theta\left( \Tt_{jk\ell}, X_{jk} \right)(p)\\
&=0
\end{align*}
for distinct $j,k,\ell$.

If \itemref{N:Aud-remix}{I:lam-exists} holds then fixing $j$  the above computation along with Lemma \ref{L:tri-perp} yields $X_{jk}^{\perp_{\nu_p}}\supset X_{jk}^{\perp_{d\theta}}$ for $k\ne j$, thus there are $\lam_k$ so that \[X_{jk}\intprod \nu_p = \lam_k X_{jk}\intprod d\theta.\]

For distinct $k_1,k_2$ not equal to $j$ we have
\begin{align*}
z_{k_2}X_{j k_1} \intprod \nu_p &= \lam_{k_1} z_{k_2} X_{j k_1} \intprod d\theta\\
z_{k_1}X_{j k_2} \intprod \nu_p &= \lam_{k_2} z_{k_1} X_{j k_2} \intprod d\theta;
\end{align*}
using \eqref{E:X-relation} this yields
\begin{align*}
z_j X_{k_1 k_2}\intprod \nu_p &= \left(  z_{k_1}X_{jk_2}-z_{k_2}X_{jk_1}\right)\intprod \nu_p\\
&= \left(  \lam_{k_2} z_{k_1} X_{j k_2}-  \lam_{k_1} z_{k_2} X_{j k_1}\right)\intprod d\theta;
\end{align*}
but repeating the above argument there is also $\lam^*$ with 
$
z_j X_{k_1 k_2}\intprod \nu_p =z_j \lam^* X_{k_1 k_2} \intprod d\theta
$
and the non-degeneracy of $d\theta$ then yields 
\[\lam_{k_2} z_{k_1} X_{j k_2}-  \lam_{k_1} z_{k_2} X_{j k_1}=\lam^* z_j  X_{k_1 k_2}
= \lam^* z_{k_1}X_{jk_2}-\lam^* z_{k_2}X_{jk_1}\]
From the independence of $X_{jk_1},X_{jk_2}$ and the fact that $z_{k_1}$ and $z_{k_2}$
are non-zero we obtain $\lam_{k_2}=\lam^*=\lam_{k_1}$.

By Lemma \ref{L:lam} \eqref{E:nec-cond} holds; by Theorem \ref{T:equiv-hd} we then have \itemref{N:Aud-remix}{I:decomp-hd}.
\end{proof}

\begin{Remark}\label{R:companion}
The proof of Theorem B
%\ref{T:Aud}  
shows that if $u$ satisfies \itemref{N:Aud-remix}{I:decomp-hd} then it also must satisfy 
\begin{align}\label{E:companion}
T_{jk}\Xt_{jk\ell}u = 0 \text{ for distinct } j,k,\ell.
\end{align}

In particular, \eqref{E:Aud2} implies the  companion condition \eqref{E:companion}.   On the other hand, the equations $TTXu=0$ and $XXTu=0$ from Theorem A
%\ref{T:PDE-for-decomp} 
do not imply each other locally (see Example 21 in [BG]) but they do imply each other when $S$ is compact and circular (see  Theorem B in [BG]). $\lozenge$
\end{Remark}

\begin{Remark} \strut
\begin{itemize}
\item[(a)]  Recall that in higher dimensions we required an additional second order vector field condition, given in \eqref{E:star}. Failure of condition \eqref{E:star} may be repaired (at least locally) by a linear change of coordinates as shown in the following proposition. 

\begin{Prop}\label{P:cov}
For $p\in S$ there is a linear transformation $T$ so that $T(S)$ satisfies \eqref{E:star} at $T(p)$ (hence also in a neighborhood of $T(p)$).
\end{Prop}

\begin{proof}  We set $z=\begin{pmatrix} z_1 \\ \vdots \\ z_n \end{pmatrix},\:
w=\begin{pmatrix} w_1 \\ \vdots \\ w_n \end{pmatrix}.$

If $M$ is an invertible square matrix then replacing $z$ by $Mz$  the transformation law from 
[Bar,\S6] tells us that $w$ is replaced by $\tensor*[^t]{M}{^{-1}}w$.  From \eqref{E:Leg-a} we know that $z$ and $w$ are non-zero; choosing $M$ from a Zariski-open dense set of matrices we may assume that all entries of the new vectors $z,w$ are non-zero.

With this in place we make a further change of variables, replacing $z$ by 
$\begin{pmatrix}
1  &  a_2 & \cdots & a_n   \\
0  &  1 &  \cdots & 0 \\
 & & \vdots & \\
0  & 0   & \cdots & 1 
\end{pmatrix} \, z$ and 
$w$  by $\begin{pmatrix}
1  &  0 & \cdots & 0   \\
-a_2  &  1 &  \cdots & 0 \\
 & & \vdots & \\
-a_n  & 0   & \cdots & 1 
\end{pmatrix} \, w$.

We find that 
\begin{quotation}
for $1<j$ the sum $z_1w_1+z_jw_j$ is replaced by $z_1w_1+z_jw_j+w_1\sum\limits_{k\notin\{1,j\}}a_kz_k$
\end{quotation}
and that 
\begin{quotation}
for $1<k<\ell$ the sum $z_kw_k+z_\ell w_\ell$
is replaced by $z_kw_k+z_\ell w_\ell-w_1\left(a_kz_k+a_\ell z_\ell\right)$.
\end{quotation}
For a Zariski-open dense set of $a_j$'s the transformed sums will all be non-zero.
% see "change-of-variable-support.nb"
\end{proof}

\item[(b)] From Theorem B
%\ref{T:Aud} 
and Proposition \ref{P:cov} it follows easily that for any relatively compact open  $U\subset S$ there are finitely many vector fields $\Xh_1,\dots,\Xh_N$ with values in $H''$ and $\Th_1,\dots,\Th_N$ with values in $H'$ so that decomposable functions on $U$ are characterized by the system
\begin{equation*}
\Xh_k\Th_ku=0 \text{ for }k=1\,\dots,N.
\end{equation*}
\end{itemize}
$\lozenge$
\end{Remark}

\section{Pluriharmonic boundary values} \label{S:plh} 

One inspiration for the current paper comes from the problem of characterizing the boundary values of pluriharmonic functions. The pluriharmonic boundary value problem has a long history, and we refer the reader to the introduction of \cite{BG} for an outline of the history. We briefly recall some key results. 

Nirenberg observed that there is no second order system of differential operators which is tangential to the boundary of the ball in $\mathbb{C}^2$ that characterizes pluriharmonic boundary values (see \cite{BG} for a discussion of this result). Bedford \cite{Bed1} provided a system of third order operators that solved the global problem for the unit ball. In higher dimensions, Audibert \cite{Aud} and Bedford \cite{Bed2} solved the global and local problems using second order systems. Bedford and Federbush  \cite{BeFe} extended these results to the case of embedded CR manifolds, and Lee extended the results to abstract CR manifolds. 

Our results parallel the results on the sphere, and we briefly recall the local results on the sphere. Define the tangential vector fields 
\begin{alignat*}{2}
L_{jk} = z_j \vf{\bar{z}_k} - z_k \vf{\bar{z}_j} &\qquad \qquad& \bar{L}_{jk}=\bar{z}_j \vf{z_k} - \bar{z}_k \vf{z_j}
\end{alignat*}
for $1\leq j, k \leq n$. Further let
\begin{align*}
\Ltd_{jk\ell} &= z_j \bar{L}_{\ell j} + z_k \bar{L}_{\ell k}.
\end{align*}

\begin{Thm}{\cite{Aud}} Suppose $S$ is a relatively open subset of $S^{2n-1}$, and $u$ is smooth on $S$. 
\begin{enumerate}
\item $u$ extends to a pluriharmonic function on a one-sided neighborhood of $S$ if and only if $$L_{jk} L_{lm} \bar{L}_{rs} u=0=\bar{L}_{jk} \bar{L}_{lm} L_{rs} u$$ for $1\leq j, k,l,m,r,s \leq n$.
\item If $n>2$, then $u$ extends to a pluriharmonic function on a one-sided neighborhood of $S$ if and only if $$L_{jk}\Ltd_{jk\ell}u=0$$ for all distinct $j, k, \ell$.
\end{enumerate}
\end{Thm}

\end{document}